\documentclass[12pt]{article}
\usepackage{amsmath, amsthm, amssymb}
\usepackage[T1]{fontenc}
\numberwithin{equation}{section}
\newtheorem{theorem}{Theorem}

\begin{document}
\author{Ajai Choudhry}
\title{New Solutions of the Tarry-Escott Problem\\ of degrees 2, 3 and 5}
\date{}
\maketitle

\begin{abstract}In this paper we obtain new parametric ideal solutions of the Tarry-Escott problem of degrees 2, 3 and 5, that is, of the diophantine systems  $\sum_{i=1}^{k+1}x_i^j=\sum_{i=1}^{k+1}y_i^j,\;j=1,\,2,\,\dots,\,k$, when $k$ is 2, 3 or 5. When $k=2$, we obtain   the complete ideal solution in terms of polynomials in six parameters $p, q, r, a, b$ and $c $ such that the common sums $\sigma_j=\sum_{i=1}^3x_i^j=\sum_{i=1}^3y_i^j$ for both $j=1$ and $j=2$ are symmetric functions of  the parameters $p, q, r$ and also symmetric functions of  the parameters $a, b, c$. When $k=3$, we  obtain a solution in terms of polynomials in four parameters $p, q, r$ and $s$ such that  the three common sums $\sigma_j= \sum_{i=1}^4x_i^j=\sum_{i=1}^4y_i^j, j=1, 2, 3$, are symmetric functions of all the  four parameters $p, q, r$ and $s$. When $k=5$, our solution is derived from the solution already obtained when $k=2$, and the common sums, defined as in the cases when $k=2$ or 3, are either 0 or  have properties similar to the case when $k=2$.  
\end{abstract}

\medskip


\section{Introduction}\label{Intro}
The Tarry-Escott problem (written briefly  as TEP)  of degree $k$ consists of finding two distinct sets
of integers $x_1,\,x_2,\,\ldots,\,x_n$, and $y_1,\,y_2,\,\ldots,\,y_n$, such that
\begin{equation}
\sum_{i=1}^nx_i^j=\sum_{i=1}^ny_i^j,\;\;\;j=1,\,2,\,\dots,\,k,
\label{tepnk}
\end{equation}
where $k$ is a given positive integer. It is well-known   that for a non-trivial
solution of (\ref{tepnk}) to exist, we must have $n \geq (k+1)$ \cite[p.\  616]{Do}. 
 Solutions of (\ref{tepnk}) with $n=k+1,$ are known as ideal solutions of the problem. 

It would be recalled that simple solutions of the TEP of degree 2 were first noticed by Goldbach and by Euler in 1750-51.
Ever since then, numerous authors have given parametric ideal solutions  of the TEP when $2 \leq k \leq 7$ and numerical ideal solutions when $k=8, 9$ or 11 (\cite{Ch}, \cite{Cho2}, \cite{Cho3}, \cite{CW}, \cite[pp. 705--713]{Di1}, \cite[pp.\ 33--57]{Gl}). Dickson \cite[pp.\  52, 55--58]{Di2} has given a complete parametric ideal solution of the TEP of degree 2 as well as a method of generating all integer solutions of the TEP of degree 3. The complete solution of the TEP is not known for any value of $k > 3$. 

In this paper,  for the TEP of degrees 2 and 3, we obtain new parametric  ideal solutions that have a remarkable symmetry that is not to be found in any of the known solutions of the TEP.
 
With reference to the diophantine system defined by  \eqref{tepnk}, for each value of the exponent $j$, let $\sigma_j$ denote the common sum of either side of \eqref{tepnk}. For the TEP of degree 2, i.e., for the diophantine system $\sum_{i=1}^3x_i^j=\sum_{i=1}^3y_i^j, j=1, 2$,  we will obtain the complete ideal solution in terms of polynomials in six parameters $p, q, r, a, b$ and $c $ such that both $\sigma_1$ and $\sigma_2$ are nonzero symmetric functions of the three parameters $p, q, r$ and also symmetric functions of the three parameters $a, b, c$. Further, the values of $x_i, y_i, i=1, 2, 3$, will all be expressible as polynomials $\phi(\xi_1, \xi_2, \ldots, \xi_6)$, where $\xi_1, \xi_2, \ldots, \xi_6$, is some permutation of the parameters $p, q, r, a, b$ and $c $. 

For  the TEP of degree 3, i.e.,  for the diophantine system $\sum_{i=1}^4x_i^j=\sum_{i=1}^4y_i^j, j=1, 2, 3$,  we will obtain an   ideal solution in terms of polynomials in four parameters $p, q, r$ and $s$ such that all the common sums $\sigma_j, j=1, 2, 3$, are nonzero symmetric functions of all the four parameters $p, q, r$ and $s$, and the values of $x_i, y_i, i=1,\ldots, 4$, are expressible as polynomials $\phi(\xi_1, \ldots, \xi_4)$, where $\xi_1, \ldots, \xi_4$, is some permutation of the parameters $p, q, r$ and $s $. 

For  the TEP of degree 5, we will derive an ideal solution in terms of six parameters $p, q, r, a, b$ and $c $ using the solution already obtained for the TEP of degree 2. In this case, while the common sums $\sigma_j, j=1, 3, 5$, are all zero,   the nonzero sums $\sigma_2$ and $\sigma_4$ are symmetric functions of $p, q$ and $ r$, as well as symmetric functions of $a, b$ and $c$.

\section{Ideal  solutions of the Tarry-Escott problem of degrees 2, 3 and 5}\label{TEP235}

\subsection{A   parametric ideal solution of the TEP of degree 2}\label{onlyTEP2} 

\begin{theorem}\label{ThTEP2}A parametric solution of the simultaneous diophantine equations,
\begin{equation}
x_1^r+x_2^r+x_3^r=y_1^r+y_2^r+y_3^r, \quad r=1, 2, \label{TEPdeg2}
\end{equation}
is given in terms of six arbitrary parameters $p, q, r, a, b$, and $ c$  by
\begin{equation}
\begin{aligned}
x_1&=\phi(p, q, r, a, b, c), \;\;&x_2&=\phi(p, q, r, b, c, a),\;\;&x_3&=\phi(p, q, r, c, a,b),\\
y_1&=\phi(p, q, r, a, c, b), \;\;&y_2&=\phi(p, q, r, c, b, a),\;\;&y_3&=\phi(p, q, r, b, a,c),
\end{aligned}
\label{solTEP2}
\end{equation}
where $\phi(f, g, h, u, v, w)=fu+gv+hw$. All integer solutions of the simultaneous diophantine equations \eqref{TEPdeg2} may be generated by taking scalar multiples of the solution given by \eqref{solTEP2}.
\end{theorem}

\begin{proof}When $x_i, y_i, i=1, 2, 3$, are defined by \eqref{solTEP2}, direct computation shows that
\begin{equation}
\begin{aligned}
\sum_{i=1}^3x_i&=\sum_{i=1}^3y_i&=&(p + q + r)(a + b + c), \\
\sum_{i=1}^3x_i^2&=\sum_{i=1}^3y_i^2&=&(p^2+q^2+r^2)(a^2 + b^2 + c^2)\\
&& & \;\;+2(pq+qr+rp)(ab+bc+ca).
\end{aligned}
\end{equation}

This proves that \eqref{solTEP2} gives a solution of the  simultaneous equations \eqref{TEPdeg2} such that the common sums $\sigma_1$ and $\sigma_2$ are  symmetric functions of  the parameters $p, q, r$, and also of  the parameters $a, b$ and $c$. 

To show that the  solution  given by \eqref{solTEP2} generates all integer solutions of the  simultaneous equations \eqref{TEPdeg2}, we will use the  following complete solution of these simultaneous equations given by Dickson \cite[p.\ 52]{Di2}:
\begin{equation}
\begin{aligned}
x_1&=AD+C, \quad &x_2&=AG+BD+C, \quad & x_3&=BG+C, \\
y_1&=AD+BG+C, & y_2&=BD+C, &y_3&=AG+C,
\end{aligned}
\label{Dicksonsol}
\end{equation}
where $A, B, C, D$ and $G$ are arbitrary parameters. 

In our solution of the simultaneous equations \eqref{TEPdeg2} given by \eqref{solTEP2}, we take $a=D,\, b=G,\, c=0,\, p = AD + AG + C,\,  q = C, \, r = BD + BG + C$,  when our solution may be written as follows:
\begin{equation}
\begin{aligned}
x_1 &= (D + G)(AD + C), \\
x_2 &= (D + G)(AG + BD + C), \\
x_3 &= (D + G)(BG + C),\\
 y_1 &= (D + G)(AD + BG + C), \\
y_2 &= (D + G)(BD + C), \\
y_3 &= (D + G)(AG + C). 
\end{aligned}
\label{solTEP2spl}
\end{equation}

The values of $x_i, y_i, i=1, 2, 3$, given by \eqref{solTEP2spl} have a common factor $D+G$, and since both the equations \eqref{TEPdeg2} are homogeneous, the common factor $D+G$  may be removed by appropriate scaling,  
and then the solution \eqref{solTEP2spl} coincides exactly with the complete  solution \eqref{Dicksonsol} of the equations \eqref{TEPdeg2} given by Dickson. It follows that the solution of the simultaneous equations \eqref{TEPdeg2} given by \eqref{solTEP2}  generates all integer solutions of these equations.
 \end{proof}

\subsection{A   parametric ideal solution of the TEP of degree 3}\label{onlyTEP3}
We will first give a theorem that gives, in parametric terms, two triads of integers with equal sums and equal products. We note that two complete solutions of the problem of finding two triads of integers with equal sums and equal products have been given independently by Choudhry \cite{Ch} and by Kelly \cite{Ke}. While the solution of this problem given  below in Theorem \ref{Thtriads} is not complete, it is much simpler and  is noteworthy for its symmetry, and   we will use it to obtain a parametric solution of the TEP of degree 3.

\begin{theorem}\label{Thtriads} A parametric solution of the diophantine system,
\begin{align}
X_1^2+X_2^2+X_3^2&=Y_1^2+Y_2^2+Y_3^2, \label{eqsumsq}\\
X_1X_2X_3&=Y_1Y_2Y_3, \label{eqprod}
\end{align}
is given  in terms of four arbitrary parameters $p, q, r$, and $ s$  by
\begin{equation}
\begin{aligned}
X_1&=\phi(p, q, r, s), \;\;&X_2&=\phi(p, r, s, q),\;\;&X_3&=\phi(p, s, q, r),\\
Y_1&=\phi(p, q, s, r), \;\;&Y_2&=\phi(p, r, q, s),\;\;&Y_3&=\phi(p, s, r, q),
\end{aligned}
\label{soltriads}
\end{equation}
where $\phi(a, b, c, d)=(ab+cd)(ac-bd)$.
\end{theorem}

\begin{proof} We begin by writing
\begin{equation}
\begin{aligned}
X_1&= f(pq + rs), \quad &X_2&= g(pr + qs), \quad &X_3&= h(ps + qr), \\
Y_1&= g(pq + rs), \quad &Y_2&= h(pr + qs), \quad &Y_3&= f(ps + qr),
\end{aligned} \label{subsxy}
\end{equation}
where $f, g, h,  p, q, r$, and $ s$, are arbitrary parameters. Now \eqref{eqprod} is identically satisfied, while \eqref{eqsumsq} reduces to

\begin{equation}
(pq+rs)^2(f^2-g^2)+(pr+qs)^2(g^2-h^2)+(ps+rq)^2(h^2-f^2)=0, \label{eqsumsq1}
\end{equation}
or, equivalently,
\begin{equation}
\{(pq+rs)^2-(ps+rq)^2\}(f^2-h^2)+\{(pr+qs)^2-(pq+rs)^2\}(g^2-h^2)=0,\label{eqsumsq2}
\end{equation}
and hence we may take,
\begin{equation}
\begin{aligned}
f^2-h^2&=(pr+qs)^2-(pq+rs)^2\\
g^2-h^2&=-\{(pq+rs)^2-(ps+rq)^2\}.
\end{aligned} 
\label{eqsumsq3}
\end{equation}
If we now take $ h=pq-rs$, it follows from  \eqref{eqsumsq3} that 
\begin{equation}
f = pr - qs, \quad g = ps - qr,
\end{equation}
and, on substituting the values of $f, g, h$ in \eqref{subsxy}, we get the solution \eqref{soltriads} stated in the theorem.
\end{proof}

\begin{theorem}\label{ThTEP3} A parametric solution of the simultaneous diophantine equations,
\begin{align}
x_1+x_2+x_3+x_4&=y_1+y_2+y_3+y_4, \label{TEPdeg3a}\\
x_1^2+x_2^2+x_3^2+x_4^2&=y_1^2+y_2^2+y_3^2+y_4^2, \label{TEPdeg3b} \\
x_1^3+x_2^3+x_3^3+x_4^3&=y_1^3+y_2^3+y_3^3+y_4^3, \label{TEPdeg3c}
\end{align}
is given in terms of four arbitrary parameters $p, q, r$, and $ s$   by
\begin{equation}
\begin{aligned}
x_1&=\phi(p, q, r, s), \quad&x_2&=\phi(p, r, s, q),\\
x_3&=\phi(p, s, q, r),  \quad&x_4&=\phi(q, r, p, s),\\
y_1&=\phi(p, q, s, r), \quad&y_2&=\phi(p, r, q, s),\\
y_3&=\phi(p, s, r, q), \quad&y_4&=\phi(q, s, p, r),
\end{aligned}
\label{solTEP3}
\end{equation}
where 
\begin{multline}
\phi(a, b, c, d)=a^2bc + abc^2 + ac^2d + acd^2 + b^2cd + bc^2d.
\end{multline}
\end{theorem}

\begin{proof} To solve the simultaneous equations \eqref{TEPdeg3a}, \eqref{TEPdeg3b} and \eqref{TEPdeg3c}, we write,
\begin{equation}
\begin{aligned}
x_1&=X_1-X_2-X_3,  &x_2&=-X_1+X_2-X_3, \\
x_3&=-X_1-X_2+X_3,  &x_4&=X_1+X_2+X_3, \\
y_1&=Y_1-Y_2-Y_3,  &y_2&=-Y_1+Y_2-Y_3,  \\
y_3&=-Y_1-Y_2+Y_3,  &y_4&=Y_1+Y_2+Y_3,
\end{aligned}
\end{equation}
when equation \eqref{TEPdeg3a} is identically satisfied while equations \eqref{TEPdeg3b} and \eqref{TEPdeg3c} reduce to equations \eqref{eqsumsq} and \eqref{eqprod} respectively. Using the solution \eqref{soltriads} of the simultaneous equations \eqref{eqsumsq} and \eqref{eqprod} given by Theorem \ref{Thtriads}, we obtain a solution of the simultaneous equations \eqref{TEPdeg3a}, \eqref{TEPdeg3b} and \eqref{TEPdeg3c} in terms of four arbitrary parameters $p, q, r$ and $s$. This solution may be written as
\begin{equation}
 \begin{aligned}
x_1&=\psi(p, q, r, s), \quad &x_2&=\psi(p, r, s, q),\\
x_3&=\psi(p, s, q, r), \quad &x_4&=\psi(q, r, p, s),\\
y_1&=\psi(p, q, s, r), \quad &y_2&=\psi(p, r, q, s), \\
y_3&=\psi(p, s, r, q), \quad &y_4&=\psi(q, s, p, r),
\end{aligned}
\label{solTEP3interim}
\end{equation}
where
\begin{equation}
\begin{aligned}
\psi(a, b, c, d)&=(bc - bd - cd)a^2 - (ac + ad - cd)b^2 \\
& \quad \quad + (ab + ad + bd)c^2 - (ab - ac + bc)d^2.
\end{aligned}
\label{defpsi}
\end{equation}

We will now use a well-known theorem according to which  if $x_i, y_i, i=1, \ldots, n$, is a solution of the diophantine system \eqref{tepnk}, then for arbitrary rational numbers $M$ and $K$, another solution of the diophantine system \eqref{tepnk} is given by $Mx_i+K, My_i+K, i=1, \ldots, n$. 

On taking $M=1/2$ and 
\begin{equation*}
K=\{p^2(qr + qs + rs) + q^2(pr + ps + rs) + r^2(pq + ps + qs)+ s^2(pq + pr + qr)\}/2,
\end{equation*}
and applying the aforesaid theorem to the solution \eqref{solTEP3interim}, we get the solution of the simultaneous equations \eqref{TEPdeg3a}, \eqref{TEPdeg3b} and \eqref{TEPdeg3c} stated in Theorem \ref{ThTEP3}.
\end{proof}

For the solution of the simultaneous equations \eqref{TEPdeg3a}, \eqref{TEPdeg3b} and \eqref{TEPdeg3c} given by \eqref{solTEP3}, we note that the values of the common sums $\sigma_j, j=1, 2, 3$, may be written as,
\begin{equation}
\begin{aligned}
\sigma_1&=2e_1e_3 - 8e_4,\\
\sigma_2&=-2e_1^2e_2e_4 + 2e_1^2e_3^2 - 8e_1e_3e_4 + 4e_2^2e_4 - 2e_2e_3^2, \\
\sigma_3&=3e_1^4e_4^2 - 3e_1^3e_2e_3e_4 + 2e_1^3e_3^3 - 12e_1^2e_3^2e_4 + 6e_1e_2^2e_3e_4 \\
& \quad- 3e_1e_2e_3^3 - 24e_2^2e_4^2 + 24e_2e_3^2e_4 - 3e_3^4 + 64e_4^3,
\end{aligned}
\end{equation}
where $e_i, i=1, \ldots, 4$, are the elementary symmetric functions defined by
\[
\begin{aligned}
e_1&=p+q+r+s, \quad &e_2&=pq + pr + ps + qr + qs + rs,\\
 e_3&=pqr + pqs + prs + qrs, \quad &e_4&=pqrs. 
\end{aligned}
\]
 
This shows that the common sums $\sigma_j, j=1, 2, 3$, are symmetric functions of the parameters $p, q, r$ and $s$.

\subsection{A parametric ideal solution of the TEP of degree 5}\label{onlyTEP5}
We will first give in Theorem \ref{Theqsums124} a parametric solution of the diophantine system $\sum_{i=1}^3x_i^j=\sum_{i=1}^3y_i^j, j=1,\,2,\,4$, and we will use this solution in Theorem \ref{ThTEP5} to obtain ideal solutions of the TEP of degree 5.

\begin{theorem}\label{Theqsums124} 
A parametric solution of the simultaneous diophantine equations,
\begin{equation}
\sum_{i=1}^3x_i^j=\sum_{i=1}^3y_i^j,\;\;\;j=1, 2, 4,
\label{eqsums124}
\end{equation}
is given in terms of six arbitrary parameters $p, q, r, a, b$, and $ c$  by
\begin{equation}
\begin{aligned}
x_1&=\psi(p, q, r, a, b, c), \;\;&x_2&=\psi(p, q, r, b, c, a),\;\;&x_3&=\psi(p, q, r, c, a,b),\\
y_1&=\psi(p, q, r, a, c, b), \;\;&y_2&=\psi(p, q, r, c, b, a),\;\;&y_3&=\psi(p, q, r, b, a,c),
\end{aligned}
\label{soleqsums124}
\end{equation}
where $\psi(f, g, h, u, v, w)=f(v - w) + g(w-u) + h(u - v)$.
\end{theorem}

\begin{proof}
In the solution \eqref{solTEP2} of the simultaneous equations \eqref{TEPdeg2}, if we replace $p, q, r$ by $r-q,\; p-r,\; q-p$, respectively, the new values of $x_i, y_i, i=1, 2, 3$, are given by \eqref{soleqsums124}, and with these values of $x_i, y_i$, direct computation gives,
\begin{equation}
\begin{aligned}
\sum_{i=1}^3x_i&=\sum_{i=1}^3y_i&=&0, \\
\sum_{i=1}^3x_i^2&=\sum_{i=1}^3y_i^2&=&2(p^2+q^2+r^2-pq-qr-rp)\\
& & &\quad \times (a^2 + b^2 + c^2-ab-bc-ca)\\
\sum_{i=1}^3x_i^4&=\sum_{i=1}^3y_i^4&=&2(p^2+q^2+r^2-pq-qr-rp)^2\\
& & & \quad  \times (a^2 + b^2 + c^2-ab-bc-ca)^2.
\end{aligned}
\label{csumsTEP5}
\end{equation}
It now follows that \eqref{soleqsums124} gives a parametric solution of the simultaneous diophantine equations \eqref{eqsums124}.
\end{proof}

\begin{theorem}\label{ThTEP5}
A parametric solution of the simultaneous diophantine equations,
\begin{equation}
\sum_{i=1}^6x_i^j=\sum_{i=1}^6y_i^j,\;\;\;j=1,\,2,\,\dots,\,5,
\label{tep5}
\end{equation}
in terms of six arbitrary parameters $p, q, r, a, b$, and $ c$ is given by  
\begin{equation}
 x_4=-x_1,\;\; x_5=-x_2,\;\;  x_6=-x_3, \;\; y_4=-y_1,\;\;  y_5=-y_2,\;\;  y_6=-y_3, 
\label{valxy456}
\end{equation}
and  the values of $x_i, y_i, i=1, 2, 3$, are defined  by \eqref{soleqsums124}. 
\end{theorem}

\begin{proof}
If we take the values of $x_i, y_i, i=4, 5, 6$, as given by \eqref{valxy456}, the relations \eqref{tep5} are identically satisfied for $k=1, 3 $ and $5$, and the diophantine system \eqref{tep5} reduces to the simultaneous diophantine equations
\begin{equation}
\sum_{i=1}^3x_i^j=\sum_{i=1}^3y_i^j,\;\;\;j=2,\,4.
\label{tep5red}
\end{equation}
Since a solution of the simultaneous equations \eqref{tep5red} is given by \eqref{soleqsums124}, it follows that a solution of the 
simultaneous diophantine equations \eqref{tep5} is as stated in the theorem. 
\end{proof}

It follows from the relations \eqref{valxy456} that for the solution of the TEP of degree 5 given by Theorem \ref{ThTEP5}, the sums $\sigma_j, j=1, 3, 5$, are all zero. Further, it follows from the relations \eqref{csumsTEP5} that the common sums $\sigma_2$ and $\sigma_4$ are symmetric functions of $p, q$ and $ r$, as well as symmetric functions of $a, b$ and $c$.

\section{Concluding remarks}
It would be interesting to find new parametric  solutions of the TEP of degrees $k$ where $k > 3$ such that all the common sums $\sigma_j, j=1, \ldots, k$, are nonzero symmetric functions of the parameters. Such solutions may be of interest even if the solutions are not ideal solutions.

More generally, it would be of interest to find parametric solutions, with similar symmetric properties, of other symmetric  diophantine  systems  of the type $\sum_{i=1}^nx_i^j=\sum_{i=1}^ny_i^j$, where the equality holds for certain values of the exponent $j$ that are not the first $k$ consecutive positive integers as in the case of the TEP.

\noindent Ajai Choudhry, 13/4 A Clay Square, Lucknow - 226001, India.

\noindent E-mail address: ajaic203@yahoo.com


\begin{thebibliography}{9}

\bibitem{Ch} J. Chernick,  Ideal Solutions of the Tarry-Escott problem, {\it Amer.
Math. Monthly} {\bf 44} (1937), 626--633.

\bibitem{Cho1} A. Choudhry,  On triads of squares with equal sums and equal products, {\it Ganita}, {\bf 49} (1998),  101--106.

\bibitem{Cho2} A. Choudhry, Ideal solutions of the Tarry-Escott problem of degrees four and five and related diophantine systems, {\it L'Enseignement Mathematique},  {\bf 49} (2003),  101--108.

\bibitem{Cho3} A. Choudhry, A new approach to the Tarry-Escott problem, {\it Int. J. Number Theory}, {\bf 13} (2017), 393--417.

\bibitem{CW} A. Choudhry and J. Wr\'oblewski, Ideal solutions of the Tarry-Escott problem of degree eleven with applications to sums of thirteenth powers, {\it Hardy-Ramanujan Journal} {\bf 31} (2008), 1--13.

\bibitem{Di1} L. E. Dickson, {\it History of the Theory of Numbers, Vol.~2}, Chelsea Publishing Company, 1952, reprint.

\bibitem{Di2} L. E. Dickson,   {\it Introduction to the theory of numbers}, Dover
Publications, New York, 1957, reprint.

\bibitem{Gl} A. Gloden,  {\it Mehrgradige Gleichungen}, Noordhoff, Groningen, 1944.

\bibitem {Do}  H. L. Dorwart   and O. E. Brown, The Tarry-Escott problem,  {\it Amer. Math.
Monthly} {\bf  44} (1937), 613--626.

\bibitem{Ke}  J. B. Kelly, Two equal sums of three squares with equal products,
{\it Amer. Math. Monthly} \textbf{98} (1991), 527--529.

\end{thebibliography}
\end{document}